\documentclass[11pt,a4paper]{amsart}
\usepackage[english]{babel}
\usepackage{bbm, amssymb, amsmath, amsfonts}
\usepackage{subcaption}

\usepackage[unicode]{hyperref}

\usepackage{graphicx, epsfig}

\usepackage{xcolor}

 \makeindex

\numberwithin{equation}{section}
\def\cal{\mathcal}

\def\Bbb{\mathbb}

\def \supp {\text{\rm supp\,}}

\def\B{{\cal B}}

\def\H{{\mathcal H}}

\def\D{{\cal D}}
\def\E{{\cal E}}
\def\F{{\cal F}}
\def\K{{\cal K}}
\def\M{{\cal M}}

\def\S{{\cal S}}

\def\cP{{\cal P}}

\def\NN{{\Bbb N}}

\def\R{{\Bbb R}}
\def\RR{{\Bbb R}}

\def\vp{{\varphi}}
\def\al{{\alpha}}

\def\la{{\lambda}}

\def\pa{{\partial}}

\def\ve{{\varepsilon}}
\def\si{{\sigma}}

\def\de{{\delta}}

\def\tr{\text{\rm tr\,}}

\def\bpm{\begin{pmatrix}}
\def\epm{\end{pmatrix}}

\def\bee{\begin{enumerate}}
\def\ee{\end{enumerate}}

\newcommand{\jb}{{\underline j}}
\newcommand{\kb}{{\underline k}}

\newtheorem{thm}{Theorem}[section]

\newtheorem{proposition}[thm]{Proposition}
\newtheorem{cor}[thm]{Corollary}
\newtheorem{lemma}[thm]{Lemma}

\newtheorem{remark}[thm]{Remark}

\theoremstyle{plain}

\begin{document}


\title[Bounds on $\Psi$DO{\tiny s}  and Fourier restriction for Schatten classes]
{Bounds on pseudodifferential operators  and Fourier restriction for Schatten classes}

\author[D. M\"uller]{Detlef M\"uller}
\address{Mathematisches Seminar, C.A.-Universit\"at Kiel,
Heinrich-Hecht-Platz 6, D-24118 Kiel, Germany} \email{{\tt
mueller@math.uni-kiel.de}}
\urladdr{{http://analysis.math.uni-kiel.de/mueller/}}

\date{18.12.24}



\begin{abstract}

As  main result, we show that a pseudodifferential operator in the Weyl calculus, whose symbol has  compact Fourier support,  lies in the Schatten class $\S^p$ if and only if its symbol lies in the Lebesgue space $L^p$ on  phase space. 

As an immediate consequence, this gives an alternative and very lucid proof  of a recent result by Luef and Samuelsen, who had discovered that for compactly supported measures $\mu,$ classical Fourier restriction estimates with respect to the measure $\mu$  are equivalent to  quantum restriction estimates for the Fourier-Wigner transform  for Schatten classes.   
  \end{abstract}

\maketitle


\tableofcontents

\thispagestyle{empty}

\section{Introduction}\label{intro}
 Suppose $\mu$ is a non-negative Radon measure on $\RR^n.$ The Fourier restriction problem associated to $\mu$,  introduced by E. M. Stein in the seventies (for measures like the Riemannian volume measure $\mu=\si$  of some smooth submanifold of 
 $\RR^n$), asks for  the range  of Lebesgue exponents
$ p$ and $q$ for which  an a priori  estimate of the form
\begin{equation}\label{Frest}
\big(\int|\widehat{f}|^{q}\,d\mu\big)^{\frac 1q}\le C\|f\|_{L^{p}(\R^n)}
\end{equation}
holds  true for   every Schwartz function   $f\in\mathcal S(\R^n),$ with a constant $C$
independent of $f.$ Here, $\hat f(\xi)=\int e^{-2\pi i \xi x} f(x)dx$ denotes the Euclidean Fourier transform of $f.$

There exists meanwhile a huge number of works on this topic, which has important applications to many other fields as well -- 
let me just give a brief overview of some of the important developments. 
The sharp range of Lebesgue exponents in dimension  $n=2$ for curves with non-vanishing curvature was determined
through work by  C. Fefferman, E. M. Stein and A. Zygmund \cite{F1}, \cite{Z}. In higher
dimension, the sharp range of $L^{ p}-L^2$ estimates  for hypersurfaces with
non-vanishing Gaussian curvature was obtained by  P. A. Tomas  and, for the end point,  by E. M. Stein \cite{To} (see also \cite{St1}).
 Such estimates, in dualized form  also known as Strichartz estimates, are of great importance also to the theory of linear and non-linear partial differential operators (compare  Strichartz \cite{Str}). Some more general classes of surfaces were
treated by A. Greenleaf \cite{Gr}. In work by I. Ikromov, M.
Kempe and D. M\"uller  \cite{ikm}  and Ikromov and M\"uller \cite{IM-uniform}, \cite{IM},
the sharp range of Tomas-Stein type    $L^{ p}-L^2$  restriction
estimates has been  determined  for a large class  of smooth, finite-type hypersurfaces,
including all analytic hypersurfaces.

The question about  general  $L^{p}-L^{q}$ restriction
estimates is nevertheless still wide open. Fourier restriction to hypersurfaces  with non-negative principal curvatures has been studied intensively by many authors. Major progress was due to J. Bourgain in the nineties (\cite{Bo1}, \cite{Bo2}, \cite{Bo3}). At the end of that decade the bilinear method was introduced (\cite{MVV1}, \cite{MVV2}, \cite{TVV} \cite{TV1}, \cite{TV2},  \cite{W2}, \cite{T2},  \cite{lv10}). A new impulse to the problem has been given with  the multilinear method (\cite{BCT}, \cite{BoG}). The best results up to date have been obtained with the polynomial partitioning method, developed by L. Guth (\cite{Gu16},
\cite{Gu18}) (see also \cite{hr19} and \cite{WA18} for recent improvements).

Initiated by ideas of G. Mockenhaupt \cite{Mo}, Fourier restriction with respect to fractal measures $\mu$ has been studied as well (see, e.g., \cite{La}).
\medskip

Typically, Fourier restriction estimates are proved by duality, i.e., by  passing to the adjoint to the Fourier restriction operator, the so-called 
 {\it Fourier extension} operator, given by
$$
\E f(\xi):=\widehat{f\,d\mu}(\xi)= \int f(x)e^{2\pi i\xi\cdot x}\,d\mu(x),
$$
and proving dual estimates of the form
\begin{equation}\label{Fext}
\|\E f\|_{L^{p'}(\R^n)}\le C\|f\|_{L^{q'}(d\mu)},
\end{equation}
where $p'$ and $q'$ denote the Lebesgue exponents conjugate to $p$ and $q$ (i.e., $1/p+1/p'=1$ and $1/q+1/q'=1$).

More recently, operator valued Fourier extension operators have been studied  by Mishra and Vemuri in \cite{MV1}, \cite{MV2} for functions on phase space $\RR^{2d},$ as well as dual Fourier restriction restriction estimates on ``quantum Euclidean space'' in \cite{HLW} by Hong, Lai and Wang.

These  Fourier extension operators are defined by means of the Weyl calculus of pseudo-differential operators. Let me here only briefly sketch the key idea of this -- for more details I refer to Folland's beautiful monograph \cite{Fo}, and Section \ref{prelim}:
\smallskip

If $a$ is a suitable  {\it symbol} (or {\it Hamiltonian} in classical mechanics)  on the phase space $\RR^{2d}=\RR^{d}\times\RR^d$ of $\RR^d,$ then we denote by $L_a$ the pseudo-differential  operator on 
$\RR^d$  associated to the symbol $a$ in the {\it Weyl calculus} (i.e., the {\it Weyl quantization} of $a$,  which is the most symmetric of all reasonable quantizations of the Hamiltonian $a$). Formally, this is given by 
$$
L_a:=\iint\hat a(x,\xi) e^{2\pi i(xD+\xi X)} dxd\xi,
$$
where $D=(D_1,\dots,D_d)$ and $X=(X_1,\dots,X_d),$ and where $X_j$ and $D_j$ denote the (unbounded) self-adjoint  operators of {\it position} $(X_jf)(x):=x_j f(x)$ and {\it momentum} $(D_jf)(x):=\frac{1}{2\pi i}\pa_{x_j}f(x)$ of quantum mechanics, which are densely defined on the Hilbert sapce $L^2(\RR^d).$ 
\smallskip

This Weyl quantization is intimately related to the Schr\"odinger representation $\pi_1$ of parameter $1$ of the $(2d+1)$-dimensional Heisenberg group $\Bbb H_d$ (as will be explained in the next section). Indeed, if  we use $(z,u)\in \RR^{2d}\times\RR$ as the coordinates of the Heisenberg group $\Bbb H_d,$ then we obtain a projective representation $\rho$ of phase space $\RR^{2d}$ by setting $\rho(z):=\pi_1(z,0).$ For suitable symbols $a$ (for instance  in $\S(\RR^{2d})$) we then have that 
\begin{equation}\label{Law}
L_a=\rho(\F_\si a)=\int_{\RR^{2d}} \F_\si a(z)\rho(z) dz,
\end{equation}
where the operator valued integral is to be understood in the weak sense, and where $\F_\si (a)$ denotes the so-called {\it symplectic Fourier transform}
$$
\F_\si a(w):=\int_{\RR^{2d}} a(z) e^{-2\pi i \si(w,z)} dz=\hat a(Jw), \qquad w\in\RR^{2d},
$$
of $a$. Here, $\si$ denotes  the symplectic form $\si((x,\xi), (x',\xi')):=x'\xi-x\xi'$ on $\RR^{2d}=\RR^d\times \RR^d,$ and 
$J:=\small{\left(
\begin{array}{ccc}
  0   &-1   \\
  1   &  0 \\
\end{array}
\right)}
$
the corresponding symplectic matrix.
Note that $\F_\si\circ \F_\si={\rm Id},$ so that we may re-write \eqref{Law} as 
\begin{equation}\label{Law2}
\int_{\RR^{2d}} u(z)\rho(z) dz=L_{\F_\si u}
\end{equation}
for suitable functions (or even distributions) $u$ on $\RR^{2d}.$ We may thus use the representation $\rho$ to define by 
$$
\rho(u):=\int_{\RR^{2d}}u(z)\rho(z) dz
$$ 
(which is sometimes, as in \cite{MV1}, \cite{MV2}, called the ``Weyl-transform'' of $u$) to define  a non-commutative, operator-valued  ``quantum'' analogue of the Euclidean Fourier transform $\hat u$ of $u.$

Correspondingly, if $\mu$ is a non-negative Radon measure on $\RR^n,$ then we may try to define by 
$$
\E_W(f):=\rho(f d\mu)=\int_{\RR^{2d}}  \rho(z) f(z) d\mu(z),
$$
a {\it quantum extension operator}, in analogy with our previous Fourier extension operator $\E.$ Note that, by setting $u:=fd\mu,$  then by \eqref{Law2}
$$
\E_W(f)=L_{\F_\si (fd\mu)}
$$
(at least formally). Since a natural quantum analogue of the space $L^{p'}(\RR^{2d})$ will  here be the Schatten class $\S^{p'}$ on the representation space $L^2(\RR^d),$ the quantum analogue of the Fourier extension estimate \eqref {Fext} would then be the estimate
\begin{equation}\label{qext}
\|\E_W(f)\|_{\S^{p'}}\le C_W \|f\|_{L^{q'}(\mu)}.
\end{equation}
$\S^p$ - estimates of $E_W(1)$ have been studied in \cite{MV1}, \cite{MV2} for surface measures $\mu$ on particular  classes of hypersurfaces in $\RR^{2d}.$ 
\smallskip

The dual estimate to  estimate \eqref{qext} turns out to be the ''quantum restriction estimate''
\begin{equation}\label{qrest}
\|\F_W(T)\|_{L^q(\mu)}\le C_W\|T\|_{\S^p},
\end{equation}
where for  any trace class operator $S\in \S^1$ on $L^2(\RR^d),$ its {\it Fourier-Wigner transform} (also called 
{\it Fourier-Weyl transform}) $\F_WS$ is defined to be the function
$$
\F_WS(z):= \tr(\rho(z)^*S), \qquad z\in\RR^{2d}.
$$
We refer to \cite{HLW}, where such kind of estimates have been studied.
\medskip

Now, in the beautiful recent preprint \cite{LSa}, Luef and Samuelsen have discovered that for {\bf compactly} supported Radon measures $\mu,$ the   quantum restriction and extension estimates \eqref{qrest} and \eqref{qext} are indeed equivalent to their classical Euclidean counterparts \eqref{Frest} and \eqref{Fext}. They prove this by showing first in \cite[Theorem1.1]{LSa} that the classical restriction estimate \eqref{Frest} and its quantum analogue \eqref{qrest} are indeed equivalent, and then proceed by duality for the remaining equivalences.
The proofs are original, but do perhaps not  really reveal the deeper reason for such a result  to hold true.

A second main result in \cite{LSa} is Theorem 1.4, which states that  for compactly supported  Radon measures $\mu,$ 
the operator $L_{\F_\si (\mu)}$ is compact if and only if $\F_\si (\mu)\in C_\infty(\RR^{2d}),$ and that $L_{\F_\si (\mu)}\in \S^p$ if and only if 
$\F_\si (\mu)\in L^p(\RR^{2d}),$ for $1\le p\le \infty.$ 
 This  theorem stands rather parallel to Theorem 1.1.

\medskip
The main goal of the present article is  to  unveil the deeper reasons for \cite[Theorem1.1]{LSa} to hold.

To this end, we shall first prove, as our {\it main result,}  in Theorems \ref{compactL} and \ref{schattenests}  a far reaching extension of \cite[Theorem1.4]{LSa}, holding  for compactly supported distributions $u$ in place of  compactly supported measures $\mu.$  Recall  that if the distribution $u$ has compact support, then $\F_\si u$ is real-analytic.

For Schatten class norms, our result  essentially says that  for compactly supported distributions $u,$ the 
$\S^p$-norm of the operator $L_{\F_\si(u)}$ and the $L^p$-norm of the function $\F_\si(u)$ are comparable:
\begin{equation}\label{psidoest}
\|L_{\F_\si(u)}\|_{\S^p}\simeq_R \|\F_\si(u)\|_{L^p(\RR^{2d})}, \qquad 1\le p\le \infty,
\end{equation}
where the implicit constants in the corresponding inequalities can be chosen to depend only on the diameter $R$ of the support of 
$u.$ For more precise information on these implicit constants we refer to Theorem \ref{schattenests}) in Section \ref{Schaest}.

\smallskip 
We believe that these results are of  interest in their own right.

\smallskip

Note, however,  that \eqref{psidoest} in particular immediately  implies the equivalence of the classical Fourier extension estimate \eqref{Fext} to  its quantum analogue \eqref{qrest}.
\smallskip

Our approach is thus very different from \cite{LSa}, concentrating first on Fourier extension estimates (as one usually also does when studying Euclidean Fourier restriction estimates), and then passing to restriction estimates by duality.  

We also note that our approach and estimates in Theorem \ref{schattenests} allow for a much better, polynomial  control of the relative sizes of the constants in Theorem \ref{restrictionS} in terms of the diameter $R$ of the support of the measure $\mu$ than in \cite{LSa}, where they grew like $e^{\pi R^2/2}$ 
(compare Remark \ref{improveconst}).
\smallskip

Main tools for our approach will be the theorem of Calder\'on-Vaillancourt  (see Theorem \ref{CaVa}), and Werner's commutative convolution product for Schatten  class operators. The latter is an important tool also in  \cite{LSa}, but in a  different way.

\subsection{Notation}\label{ss:notation}

If $A$ and $B$ are two nonnegative numbers, we write $A \lesssim B$ to indicate that there exists a constant $C>0$ such that $A \leq C B$. We also write $A \simeq B$ to denote the conjunction of $A \lesssim B$ and $B \lesssim A$. Subscripted variants such as $\lesssim_R$ and 
$\simeq_R$ are used to indicate that the implicit constants may depend on a parameter $R$. We shall also use the 
``variable constant'' notation, meaning that a constant $C$ may be of a different size in different lines.

For $x,y\in \RR^n,$ we shall denote by $xy$ the  Euclidean inner product given by $xy:=\sum_j x_jy_j.$ The hermitian inner product on $L^2(\RR^d)$ will be written as 
$$
(f,g):=\int_{\RR^d} f(x)\overline {g(x)} dx.
$$ 
If $u\in \S'(\RR^n)$ is a tempered distribution and  $\vp\in\S(\RR^n)$ a Schwartz function, then we shall write $\langle u, \vp\rangle:=u(\vp).$  Accordingly, for suitable functions $f,g$ on $\RR^d,$ we shall write 
$$
\langle f,g\rangle:=\int_{\RR^d} f(x)g(x) dx.
$$

\section{Background and Preliminaries}\label{prelim}

An important reference for the discussions in this section will be \cite{Fo}. However, since the authors in  \cite{LSa} use slightly different coordinates than in \cite{Fo}, we shall rather keep to the  notation in the latter paper. Let us begin by re-calling a number of notions.

\medskip
The {\it Heisenberg group} $\Bbb H_d$ is   $\RR^d\times \RR^d\times \RR$ as a manifold, but endowed with the non-commutative product 
$$
(x,\xi,u)(x',\xi',u'):=(x+x', \xi+\xi', u+u'+\frac {x'\xi-x\xi'}2),
$$
which turns $\Bbb H_d$ into a 2-step nilpotent Lie group.

We also use $z=(x,\xi)\in\RR^d\times \RR^d= \RR^{2d}$ for coordinates of the  {\it phase space}  $\RR^{2d}$, which is endowed as usually with the {\it symplectic form}
$$
\si((x,\xi), (x',\xi')):=x'\xi-x\xi'.
$$

\medskip

The {\it Schr\"odinger represention} $\pi_1$ of $\Bbb H_d$  on $L^2(\RR^d)$ with central character $u\mapsto e^{2\pi i u}$ is given by
$$
[\pi_1(x,\xi,u) g](t):= e^{2\pi i u}e^{-\pi ix\xi+2\pi it\xi}g(t-x), \qquad g\in L^2(\RR^d).
$$

\medskip

By restricting to $\RR^{2d},$ we obtain the {\it projective Schr\"odinger representation} $\rho$ of $\RR^{2d}$ on $L^2(\RR^d):$
$$
[\rho(x,\xi) g](t):= [\pi_1(x,\xi,0) g](t):=e^{-\pi ix\xi+2\pi it\xi}g(t-x), \qquad g\in L^2(\RR^d),
$$
which satisfies
\begin{equation}\label{proj}
\rho(z+z')= e^{i\pi \si(z,z')}\rho(z)\rho(z').
\end{equation}

The integrated representation of $L^1(\RR^{2d}, dz)$ is defined by 
$$
\rho(F)g:=\int_{\RR^{2d}} F(z) \rho(z) g \,dz, \qquad F\in L^1(\RR^{2d}), g\in L^2(\RR^d).
$$
where the integral is to be read as a Bochner integral.
Then $\rho(F)$ is an integral operator
$$
[\rho(F)g](t)=\int_{\RR^{d}} k_F(t,x) g(x) dx,
$$
with integral kernel
\begin{equation}\label{intker}
k_F(t,x):=\int F(t-x,\xi)e^{\pi i(t+x)\xi}d\xi=\F^\xi (F)(t-x,-\frac{t+x} 2),
\end{equation}
where $\F^\xi$ denotes the partial (Euclidean) Fourier transform with respect to $\xi.$ 
\smallskip

Conversely, if $K$ is a suitable integral kernel, then we shall denote by 
$$
[T_K g](x):= \int_{\RR^{d}} K(t,x) g(x) dx,
$$
the corresponding integral operator, so that in particular $\rho(F)=T_{k_F}.$
\smallskip

In particular,  if $F\in \S(\RR^{2d}),$ then  the operator $\rho(F)$ is of trace class, and  by Fourier inversion we get 
\begin{equation}\label{trace}
\tr\rho(F)=\int k_F(x,x) dx= \int \F^\xi(F)(0,-x) dx= F(0,0).
\end{equation}

We next recall that the  previous definition of  $\rho(F)$  can  be extended  to tempered distributions (compare  \cite[Theorem (1.30)]{Fo}; note  that Folland uses a slightly different, but equivalent,  definition of the projective representation $\rho$ -- see Remark \ref{FoLSa}):
\smallskip

if $u\in \S'(\RR^n),$ then  $\rho(u)$ is the  continuous linear operator $\rho(u):\S(\RR^d)\to \S'(\RR^d),$ whose 
 Schwartz kernel  $k_u$ is given by the tempered distribution 
\begin{equation}\label{ku}
k_u(t,x):=(\F^\xi u)(t-x,-\frac{t+x} 2),
\end{equation}
where the  linear change of coordinates in the argument of $\F^\xi u$  is to be interpreted in the sense of distributions. Recall that this means that 
$$
\langle \rho(u)\psi,\vp\rangle=\langle k_u,\vp\otimes\psi\rangle \qquad\text{for all} \ \vp,\psi\in \S(\RR^d).
$$

One  checks easily that the linear mapping $\cP:\S'(\RR^{2d})\to \S'(\RR^{2d}),$ 
\begin{equation}\label{Pdef}
\cP:u\mapsto k_u,
\end{equation}
given by \eqref{ku} is a topological isomorphism, which is unitary as an  operator on $L^2(\RR^{2d}).$  Moreover, when restricted to $\S(\RR^{2d}),$ $\cP$ becomes a topological automorphism of the F\'rechet space $\S(\RR^{2d}).$ Also in this situation, we shall write $\rho(u)=T_{k_u}.$ 
\smallskip

In particular, note that, by some easy computations, we find that for  $u\in \S'(\RR^{2d})$ and  $F\in \S(\RR^{2d})$  the following identity holds true:
\begin{equation}\label{cP}
\langle k_u,k_F\rangle=\langle u,P_\xi F\rangle,
\end{equation}
where $P_\xi F(x,\xi):=F(x,-\xi).$

\medskip
If we introduce as usually the {\it twisted convolution product} for  suitable functions on $\RR^{2d}$ by setting 
$$
F\times G(z):= \int_{\RR^{2d}} F(z')G(z-z') e^{\pi i \si(z',z-z')}dz' \qquad (\text{say} \ F,G\in L^1(\RR^{2d})),
$$
then 
\begin{equation}\label{rhotwist}
\rho(F\times G)=\rho(F)\rho(G).
\end{equation}

\medskip
The {\it symplectic Fourier transform} $\F_\si (F)$ of $F\in L^1(\RR^{2d})$ is defined by
$$
\F_\si F(w):=\int_{\RR^{2d}} F(z) e^{-2\pi i \si(w,z)} dz, \qquad w\in\RR^{2d}.
$$
Note that $\F_\si\circ \F_\si={\rm Id}.$

\medskip
By $\K$ we  denote the space of all compact linear operators on  the Hilbert space $L^2(\RR^d).$ $\K$ forms a closed, two-sided ideal in the 
$C^*$-algebra $\B(L^2(\RR^d))$ of all bounded linear operators on $L^2(\RR^d).$ 

For $1\le p< \infty,$ we denote by  $\S^p\subset \K$ the {\it $p$-Schatten class}  (which forms a non-commutative analogue of the $\ell^p$ space; see, e.g., \cite{Si}). The space $\S^p$ is endowed with the  norm
$$
\|T\|_{\S^p}:= \tr (|T|^p)^{\frac 1p}=\tr\big((T^*T)^{\frac p2}\big)^{\frac 1p},  \qquad T\in \S^p,
$$
where $\tr S$ denotes the trace of a trace class operator $S.$
 Recall that  $(\S^p,\|\cdot\|_{\S^p})$ is a  Banach space, which forms a two-sided ideal in the algebra $\K.$ Indeed, $\S^p$ is even a two-sided ideal in $\B(L^2(\RR^d)),$ and if $T\in\S^p$ and $A,B\in \B(L^2(\RR^d)),$ then
 \begin{equation}\label{ATB}
\|ATB\|_{\S^p}\le \|A\| \|T\|_{\S^p}\|B\|.
\end{equation}
 
 It is useful to also define $\S^\infty:=\B(L^2(\RR^d)),$ endowed with the operator norm $\|\cdot\|.$ Then 
 $$
 \S^1\subset \S^p\subset \S^q, \qquad \text{if} \ 1\le p\le q\le \infty.
 $$
 Moreover, $\S^1=\S^1\cap \S^\infty$ is dense in $\S^p$ for $1\le p<\infty,$ and 
 $$
 \S^p=[\S^1,\S^\infty]_\theta
 $$ 
 is the interpolation space between $\S^1$ and $\S^\infty$ given by the complex interpolation method, with parameter $\theta$ defined by  $1/p=1-\theta$ (see \cite{Si}, \cite{BL}).
\medskip

Note: The subspace of all $\rho(F), F\in \S(\RR^{2d}),$ is dense in $\S^p$ for $1\le p<\infty$ (see, e.g., Lemma \ref{rhoSdense}).
\medskip

Next, we recall 
{\it Werner's convolution product} for Schatten class operators in QM-Harmonic Analysis:

For $T_1,T_2\in \S^1,$ one puts
$$
T_1\star T_2(w):=\tr[\rho(-w)T_1\rho(w)\, PT_2P],\qquad w\in\RR^{2d},
$$
where $P$ denotes the involution 
$$
Pg(t):=\check  g:=g(-t),\qquad  g\in L^2(\RR^d).
$$

\smallskip

This convolution is associative and commutative and satisfies the following analogue of Young's inequality (see \cite[Prop. 3.2]{W},\cite[Prop. 4.2]{LSk}).

\begin{proposition}[Young's inequality for Werner's convolution]\label{youngw}
Let $1\le p,q,r\le \infty$ so that $1/p+1/q=1+1/r.$ Then for $T_1,T_2\in \S^1,$ we have
$$
\|T_1\star T_2\|_{L^r}\le \|T_1\|_{\S^p} \|T_2\|_{\S^q}.
$$
Consequently, Werner's convolution extends by continuity to a bounded bilinear mapping 
$\star: \S^p\times \S^q\to L^r(\RR^{2d}),$ if $p,q<\infty.$ Moreover, if $p=\infty$ (or $q=\infty$),  then $q=1$ (respectively  $p=1$), and the convolution product 
as defined above  still makes sense and the convolution inequality remains valid as well.
\end{proposition}

\medskip
Let us compute $T_1\star T_2$  for $T_j=\rho(F_j), j=1,2,$ where $F_1,F_2\in \S(\RR^{2d}):$

\smallskip
Note to this end that for $w\in \RR^{2d},$  $\rho(w)=\rho(\delta_w),$ where $\delta_w$ denotes the point measure at $w.$ Thus, by \eqref{rhotwist},
$T_1\rho(w)=\rho(F_1\times \delta_w),$ where $F_1\times \delta_w(z)=F_1(z-w)e^{\pi i \si(z,w)}.$ 
Then one easily computes that 
$$
\delta_{-w}\times F_1\times\delta_w(z)=F_1(z) e^{2\pi i \si(z,w)}.
$$
Moreover, one checks that $P\rho(z)=\rho(-z) P,$
which implies that 
$$
PT_2=P\rho(F_2)=\rho(\check F_2)P,
$$
hence
$$
PT_2P=\rho(\check F_2),
$$
where $\check F_2(z):=F_2(-z).$ Thus,  finally
\begin{eqnarray*}
\delta_{-w}\times F_1\times\delta_w\times \check F_2(0)&=&\int F_1(z') e^{2\pi i \si(z',w)}\check F_2(0-z')e^{\pi i \si(z',0-z')}dz'\\
&=&\int (F_1F_2)(z')e^{2\pi i \si(z',w)} dz'=\F_\si(F_1F_2)(w).
\end{eqnarray*}
In combination with \eqref{trace}, we eventually find that for $F_1,F_2\in \S(\RR^{2d}),$
\begin{equation}\label{star}
\rho(F_1)\star \rho(F_2)=\F_\si(F_1F_2)=\F_\si(F_1)*\F_\si(F_2),
\end{equation}
where $*$ denotes here the Euclidean convolution on $\RR^{2d}.$ From this identity (and Lemma \ref{rhoSdense}), it becomes also evident why  Werner's convolution is commutative.
\medskip

Finally, if $a\in\S(\RR^{2d}),$  we recall that the   {\it Weyl quantization} $L_a$ of the symbol $a$ is defined as the bounded linear operator 
 $$
 L_a:=\int_{\RR^{2d}} \F_\si(a)(z)\rho(z) dz=\rho(\F_\si(a))\in \B(L^2(\RR^d)).
 $$
 \smallskip
 
Again, by \eqref{ku}, the definition of $L_a$ extends to tempered distributions:  if $a\in\S'(\RR^{2d}),$ then $L_a$ 
is the continuous linear operator 
$L_a:\S(\RR^d)\to \S'(\RR^d),$ whose 
 Schwartz kernel  $K_a$ is given by the tempered distribution 
$K_a(t,x):=k_{\F_\si(a)} (t,x).$ One then easily computes from \eqref{ku} that
\begin{equation}\label{Ka}
K_a(t,x)=(\F^\xi a)\big(\frac {x+t}2, x-t\big).
\end{equation}

\begin{remark}\label{FoLSa}
If we denote by $\tilde \pi_1$ and $\tilde\rho$ the analogues of the Schr\"odinger representation $\pi_1$ and $\rho$ as defined by Folland in \cite{Fo}, then one easily checks that $\tilde \pi_1(Jz,u)=\pi_1(z,u),$ and thus $\tilde \rho(Jz)=\rho(z).$ Since $J$ is symplectic, the mapping $(z,u)\mapsto (Jz,u)$ is an automorphism of the Heisenberg group which fixes its center, so that, by the Stone- von Neumann theorem, the representations $\tilde \pi_1$ and $\pi_1$ are unitarily equivalent, and the same is true of the projective representations $\tilde\rho$ and $\rho.$ Moreover, we have
$$
L_a=\iint\hat a(x,\xi)\tilde\rho(x,\xi) dxd\xi=\iint\F_\si a(x,\xi)\rho(x,\xi) dxd\xi.
$$
This allows, if we which, to easily switch between the notation in \cite{Fo} and the one in \cite{LSa} (which we are using here).
\end{remark}

 A crucial tool for us will be the theorem of Calder\'on-Vaillancourt (\cite[Theorem 2.73]{Fo}):
 \begin{thm}[Calder\'on-Vaillancourt]\label{CaVa}
Suppose $a$ is of class $C^{2d+1}$ on $\RR^{2d},$ and  that
$$
\|a\|_{C^{2d+1}}:=\sum\limits_{|\al|+|\beta|\le 2d+1}\|\pa_x^\al \pa_\xi^\beta a\|_\infty <\infty.
$$
Then $L_a$ is bounded on $L^2(\RR^d).$ Moreover, there exists a constant $C>0$ such that 
$$
\|L_a\|\le C\|a\|_{C^{2d+1}}
$$
whenever $\|a\|_{C^{2d+1}}<\infty.$
\end{thm}

\section{Auxiliary results}\label{aux}

Our first auxiliary result is well-known, at least as folklore, and I shall only give a brief sketch of the proof.
First, we recall some well-known facts about Hermite functions (compare \cite[Ch. 7]{Fo}, but note that Folland works in coordinates $y$ on $\RR^n$ which are related to the coordinates $x$ used here by $x=\sqrt{2\pi}y$): let 
$$
\H_1:=-\frac {d^2}{dx^2}+x^2
$$
denote the {\it Hermite operator} on $\RR,$ and let 
\begin{equation}\label{hermite}
 \tilde h_k(x):=(-1)^k e^{x^2/2} \frac {d^k}{dx^k} e^{-x^2}=H_k(x)\, e^{-x^2/2},\qquad k\in\NN,
\end{equation}
denote the k-th classical {\it Hermite function} ($H_k$ denoting the Hermite polynomial of degree $k$).
We normalize these Hermite functions  by putting 
$$
h_k:=\frac 1{\sqrt{2^k k! \sqrt{\pi}}}\, \tilde h_k, \qquad k\in \NN.
$$ 
 Then the family $\{h_k\}_{k\in \NN}$ forms a countable complete   orthonormal  basis of the Hilbert space  $L^2(\RR),$ consisting of eigenfunctions of the Hermite operator:
\begin{equation}\label{EVs}
\H_1h_k=(2k+1) h_k, \qquad k\in\NN.
\end{equation}
Similarly, on $\RR^n$ we can consider the family 
$$
h_{\kb}(x):=h_{k_1}(x_1)\cdots h_{k_n}(x_n), \qquad \kb=(k_1,\dots,k_n)\in\NN^n.
$$
Again, these higher-dimensional Hermite functions form a complete  orthonormal  basis of the Hilbert space  $L^2(\RR^n), $ each of them being an eigenfunction of the $n$-dimensional Hermite operator
 $
\H_n:=\sum\limits_{j=1}^n(-\frac {\pa^2} {{\pa x_j}^2}+x_j^2)=-\Delta+|x|^2:
$
\begin{equation}\label{eigenHF}
\H_n h_\kb=(n+2|\kb|)h_\kb,  \qquad \qquad \kb\in\NN^n.
\end{equation}
We remark that, at the same time, the  $h_\kb$'s are eigenfunctions of the Euclidean Fourier transform, if we define the Fourier transform in this section by $\hat f(\xi):=\int f(x) e^{-i\xi x} dx:$
$$
\widehat  h_\kb= (-i)^{|\kb|} \, h_\kb
$$
(note that if we express the $h_\kb$ in the coordinates $y=(2\pi)^{-1/2}x$ used by Folland, they become eigenfunctions with respect to the Fourier transform as defined in the previous sections).
\smallskip

If $f\in L^2(\RR^n),$ we shall call
$$
\breve f(\kb):=(f, h_\kb)
$$
the $\kb$-th {\it Hermite coefficient} of $f.$ If $\vp\in \S(\RR^n),$ then we define its Schwartz-norm $\|\vp\|_{(N,2)}$ by
$$
\|\vp\|_{(N,2)}:=\sum\limits_{|\al|+|\beta|\le N}\|\pa^\al(x^\beta\vp)\|_{L^2(\RR^n)}.
$$

\begin{lemma}\label{charS}
If $\vp\in\S=\S(\RR^n)$ is a  Schwartz-funktion,  then its Fourier-Hermite-transform $\breve\vp:\kb\mapsto \breve\vp(\kb)$ is rapidly decaying on $\NN^n.$ 
More precisely,  for every $N\in \NN,$ 
\begin{equation}\label{rapid}
|\breve \vp(\kb)|\le C_N(n+|\kb|)^{-N}\|\vp\|_{(2N,2)}.
\end{equation}
 Moreover, the  expansion 
\begin{equation}\label{expansionS}
\vp=\sum\limits_{\kb\in\NN^n} \breve\vp(\kb) h_{\kb}
\end{equation}
of $\vp$ with respect to the Hermite basis is convergent with respect to the topology of   the Fr\'echet space $\S(\RR^n).$
\end{lemma}

\begin{proof} By \eqref{eigenHF}, we have that for every $N\in \NN$ 
\begin{eqnarray*}
\breve \vp(\kb)=(\vp, h_\kb)=(n+2|\kb|)^{-N} (\vp,\H_n^N h_\kb)=(n+2|\kb|)^{-N} (\H_n^N \vp,h_\kb),
\end{eqnarray*}
 so that 
 $$
 |\breve \vp(\kb)|\le(n+2|\kb|)^{-N} \|\H_n^N \vp\|_2.
 $$ 
This easily implies  \eqref{rapid}.
  \smallskip
  
 With a little more effort one can also show that for every continuous semi-norm $p$ on $\S,$  $p(h_{\kb})$ will only grow polynomially in $\kb,$ i.e., 
 $$
p(h_\kb)\le C_p(n+2|\kb|)^{N_p}\qquad \text{for every} \ \kb\in\NN^n,
$$
for suitable constants $C_p$ and $N_p.$ This follows indeed easily by induction on $|\kb|$  by means of  the classical recurrence relations for Hermite functions (compare also the identities (1.82) in \cite{Fo}).
The proof of the lemma now follows from these estimates.
\end{proof}

\begin{lemma}\label{S1est}
If $\vp\in\S(\RR^{2d}),$ then $L_\vp\in \S^1,$ and 
\begin{equation}\label{s1esti}
\|L_\vp\|_{\S^1}\le C \|\F_\si(\vp)\|_{(4d+2,2)}.
\end{equation}

\end{lemma}
\begin{proof} 
By \eqref{Ka},   the integral kernel of $L_\vp$ is given by
\begin{equation}\label{Kphi}
K_\vp(t,x):=k_{\F_\si\vp}(t,x)=(\F^\xi \vp)(\frac {x+t}2, x-t), \qquad t,x\in \RR^d.
\end{equation}

If we regard $K_\vp$ as a function on $\RR^{2d},$ then it is easily seen that
\begin{equation}\label{equiv1}
\|K_\vp\|_{(N,2)}\simeq \|\vp\|_{(N,2)}\simeq \|\F_\si(\vp)\|_{(N,2)}.
\end{equation}
Indeed, since the linear change of variables $(t,x)\mapsto (\frac {x+t}2, x-t)$ is invertible, we first observe that 
$\|K_\vp\|_{(N,2)}\simeq \|\F^\xi (\vp)\|_{(N,2)}.$ Next, by Plancherel's theorem (on $\RR^d$), we find that 
$\|\F^\xi (\vp)\|_{(N,2)}\simeq \|\vp\|_{(N,2)},$ and then, by rotation through the matrix $J$ and Plancherel's theorem  on $\RR^{2d}$  finally that
$\|\vp\|_{(N,2)}\simeq \|\F_\si(\vp)\|_{(N,2)}.$
\medskip

Next, according to Lemma \ref{charS}, we may expand 
\begin{equation}\label{kphixp}
K_\vp(t,x)=\sum\limits_{\jb,\kb\in\NN^d} \breve K_\vp(\jb,\kb) h_{\jb}(t)h_{\kb}(x),
\end{equation}
where, for every $N\in\NN,$
\begin{equation}\label{rapidK}
|\breve K_\vp(\jb,\kb)|\le C_N(2d+2|\jb|+2|\kb|)^{-N}\|K_\vp\|_{(2N,2)}.
\end{equation}
But, if $N\ge 2d+1,$ then $\sum \limits_{(\jb,\kb)\in\NN^{2d}}(2d+2|\jb|+2|\kb|)^{-N}<\infty.$
And, if we identify $h_{\jb}\otimes h_{\kb}$ with the one-dimensional projector $T_{h_{\jb}\otimes h_{\kb}}:g\mapsto (g,h_{\kb})h_{\jb},$ then \eqref{kphixp} shows that 
$$
L_\vp=\sum\limits_{\jb,\kb\in\NN^d} \breve K_\vp(\jb,\kb) T_{h_{\jb}\otimes h_{\kb}}.
$$ 
Since  $T_{h_{\jb}\otimes h_{\kb}}$ is a one-dimensional projector of trace-class norm  $\|T_{h_{\jb}\otimes h_{\kb}}\|_{\S^1}=1,$ in view of \eqref{equiv1} we may then estimate
\begin{eqnarray*}
\|L_\vp\|_{\S^1}\le\sum\limits_{\jb,\kb\in\NN^d} |\breve K_\vp(\jb,\kb)|\le C \|K_\vp\|_{(4d+2,2)}\simeq \|\F_\si(\vp)\|_{(4d+2,2)}.
\end{eqnarray*}
\end{proof}

If $a\in\S'(\RR^{2d}),$ then we say that $L_a\in \S^p,$ if $L_a$ maps $\S(\RR^d)$ continuously into $L^2(\RR^d)$ with respect to the $L^2$-norm, and if in addition its continuous extension to a bounded  linear operator on $\S(\RR^d)$ is compact and lies in $\S^p.$ 

Moreover, if $F$ is a function (or distribution)  on $\RR^{2d},$ and if $w\in \RR^{2d}$, then we write $(\la_wF)(z):=F(z-w)$ for the  translation of $F$ by $w.$ 

\begin{lemma}\label{Sptrans}
Let  $1\le p\le\infty,$  and suppose  $a\in\S'(\RR^{2d})$  is such that $L_a\in \S^p.$ Then  for every $w\in \RR^{2d},$ we have 
\begin{equation}\label{law}
\|L_{\la_w a}\|_{\S^p}=\|L_{a}\|_{\S^p}.
\end{equation}
\end{lemma}

\begin{proof} 
To prove \eqref{law}, assume first that $w=(y,0)\in\RR^d\times \RR^d.$ Then by \eqref{Ka} (again in the sense of distributions),
$$
K_{\la_w a}(t,x)=(\F^\xi a)(\frac {x+t}2-y, x-t), 
$$
so that $K_{\la_w a}(t+y,x+y)=K_a(t,x),$ i.e.,
$$
L_{\la_w a}=\la_y\circ L_a\circ \la_{-y},
$$
where  the operator $\la_y$ of translation by $y$ is unitary on  $L^2(\RR^d).$  
\smallskip

But, if $T\in \S^p,$ and  if $U$ is any unitary operator on $L^2(\RR^d),$
then it is well-known that 
$$
\|UTU^{-1}\|_{\S^p}=\|T\|_{\S^p}.
$$ 
Here is a brief sketch  of a proof: Let $A:=UTU^{-1}.$ Then clearly $|A|^2=A^*A=UT^* TU^{-1}=U|T|^2U^{-1},$ and thus by standard functional calculus for self-adjoint  operators, we have that $|A|^p=U|T|^p U^{-1}.$ This immediately implies the preceding identity, since 
$\|T\|^p_{\S^p}=\tr |T|^p.$

\smallskip
We conclude in particular that \eqref{law} holds for $w=(y,0).$
\smallskip

It remains to show the same for $w=(0,\eta)\in\RR^d\times \RR^d.$ But, here 
$$
K_{\la_w a}(t,x)=e^{2\pi i \eta t}K_a(t,x) e^{-2\pi i \eta x,}
$$
hence 
$$
L_{\la_w a}=U L_a U^{-1},
$$
where $U$ denotes the unitary operator of multiplication with $e^{2\pi i \eta x}$ on  $L^2(\RR^d).$ Thus, \eqref{law} holds in this case as well.
\end{proof} 

If $u\in \S'(\RR^{2d}),$ and if  $1\le p\le \infty,$ then (with a slight abuse of notation)  we shall say that $\rho(u)$ lies in the Schatten class $S^p,$ if the operator $\rho(u)$ maps $\S(\RR^d)$ boundedly  into $L^2(\RR^d)$ with respect to the $L^2$-norm in such a way that its unique extension to a bounded linear operator on $L^2(\RR^d)$ lies in $S^p.$

Conversely, note that  if $T\in S^\infty=\B(L^2(\RR^{d})),$ then $T$ maps the space $\S(\RR^d)$ continuously into $\S'(\RR^d),$ and thus, by the Schwartz kernel theorem, there exists a unique tempered distribution $k\in\S(\RR^{2d})$ so that 
$$
\langle T\psi,\vp\rangle= \langle k,\vp\otimes\psi \rangle\qquad \text{for all} \ \vp,\psi\in \S(\RR^d).
$$
Consequently,  there is a unique tempered distribution $u=u_T$ in $\S'(\RR^{2d})$ so that $k=k_u,$ hence 
$T=\rho(u),$ namely $u:=\cP^{-1}(k),$ with $\cP$ as defined in \eqref{Pdef}.
\smallskip

\begin{lemma}\label{wernerg}
Let $T=\rho(u) \in S^\infty,$  where $u=u_T\in\S'(\RR^{2d})$ is defined as before.  Then, for every $\phi\in\S(\RR^{2d}),$ the following analogue of identity \eqref{star} holds true pointwise:
 \begin{equation}\label{star2}
T\star\rho(\phi)=\rho(u)\star \rho(\phi)=\F_\si(\phi u)=\F_\si(u)*\F_\si(\phi).
\end{equation}
 \end{lemma}
 Note here  that  $\F_\si(u)*\F_\si(\phi)$ is a continuous function, since $\F_\si(\phi)\in\S$ and  $\F_\si(u)\in\S'.$ Moreover, since 
 $\F_\si(u)*\F_\si(\phi)=\langle \F_\si(u),\la_w (\F_\si(\phi)\check)\rangle,$ one easily sees that 
 $$
 \F_\si(\phi u)(w)=\langle u e^{-2\pi i\si(w,\cdot)},\phi\rangle.
 $$
 \begin{proof} Since $\phi\in\S(\RR^{2d}),$ also $k_\phi=\cP(\phi)\in\S(\RR^{2d}).$   Arguing in a similar way as in the proof of Lemma \ref{S1est},  we may thus expand $k_\phi$ by means of Lemma \ref{charS} as 
\begin{equation}\label{kFexp}
k_\phi(t,x)=\sum\limits_{\jb,\kb\in\NN^d} \breve k_\phi(\jb,\kb) h_{\jb}(t)h_{\kb}(x),
\end{equation}
where, for any $N\in\NN,$ 
\begin{equation}\label{rapidk}
|\breve k_\phi(\jb,\kb)|\le C_N(2d+2|\jb|+2|\kb|)^{-N}\|k_\phi\|_{(2N,2)}.
\end{equation}
The series \eqref{kFexp} is convergent in the topology of $\S,$ so if we put 
$$
K_n:=\sum\limits_{\{(\jb,\kb): |\jb|+|\kb|\le n\}} \breve k_\phi(\jb,\kb) h_{\jb}\otimes h_{\kb},\qquad n\in\NN,
$$
then $k_\phi=\lim K_n$ in $\S.$ Correspondingly, if we set 
$$
\phi_n:=\sum\limits_{\{(\jb,\kb): |\jb|+|\kb|\le n\}} \breve k_\phi(\jb,\kb) \cP^{-1}(h_{\jb}\otimes h_{\kb})\in \S(\RR^{2d}),
$$
then $\phi=\lim\limits_{n\to \infty} \phi_n$ in $\S.$ Since $\F_\si(u)*\F_\si(\phi)=\langle \F_\si(u),\la_w (\F_\si(\phi)\check)\rangle,$ this implies that
\begin{equation}\label{FntoF}
\F_\si(u)*\F_\si(\phi)(w)=\lim\limits_{n\to \infty}\\F_\si(u)*\F_\si(\phi_n).
\end{equation}

Next, by  Lemma \ref{S1est}, $\rho(\phi)$ lies in the operator ideal $\S^1$ of $\B(L^2(\RR^d))=\S^\infty$,  and since 
 $T_{\|h_{\jb}\otimes h_{\kb}}\|_{\S^1}=1,$ \eqref{rapidk} also implies that 
$$
\|\rho(\phi)-\rho(\phi_n)\|_{S^1}\to 0 \quad {as}\ n\to \infty.
$$
Recalling that for any $w\in\RR^{2d},$
$$
T\star\rho(\phi)(w):=\tr[\rho(-w)T\rho(w)\, P\rho(\phi)P],
$$
we conclude that
\begin{equation}\label{FnTtoFT}
T\star\rho(\phi)(w)= \lim\limits_{n\to \infty} T\star\rho(\phi_n)(w).
\end{equation}

By \eqref{FntoF} and \eqref{FnTtoFT}, it will thus suffice  to prove \eqref{star2} for the case where $k_\phi=\phi_1\otimes \phi_2,$ i.e., $T=T_{\phi_1\otimes \phi_2},$ with $\phi_1,\phi_2\in\S(\RR^d)$ being real-valued Schwartz functions.
\smallskip

However, straight-forward computations show that for $f,g\in L^2(\RR^d),$ the Schwartz kernel of the operator $T_{f\otimes g}$ is given by 
\begin{equation}\label{utensor}
f\otimes g=k_{A(f,\overline g)},
\end{equation}
where 
\begin{equation}\label{crosFW}
A(f,g)(x,\xi):=\cP^{-1}(f\otimes g)(x,\xi)=(f,\rho(x,\xi)g)=\int_{\RR^d} f(t+\frac x2)\overline{g(t-\frac x2)} e^{-2\pi i\xi t}dt
\end{equation}
is what is called the {\it cross ambiguity function} of $f$ and $g$ in \cite{LSa} (in the coordinates chosen in \cite{Fo} it corresponds to what Folland denotes  as the  {\it Fourier-Wigner transform} $V(f,g)$ of $f$ and $g$). 

Note also that 
\begin{equation}\label{tracetensor}
\tr(T_{f\otimes g})=\int f(x)g(x) dx.
\end{equation}
 
In particular,  if $k_\phi=\phi_1\otimes\phi_2,$ with $\phi_1,\phi_2$ real-valued Schwartz functions, then 
$\phi=A(\phi_1,\phi_2).$ 

\smallskip

Recall next from Section \ref{intro} that for any Schwartz function $F\in \S(\RR^{2d})$ we have 
$$
\rho(-w)\rho(F)\rho(-w)=\rho(\de_{-w}\times F\times \de_w)=\rho(F e^{2\pi i\si(\cdot,w)}).
$$
Since $\S$ lies dense in $\S'$ with respect to the weak*- topology, this identity remains valid also for tempered distributions, so that we have in particular that 
$$
\rho(-w)\rho(u)\rho(-w)=\rho(\de_{-w}\times u\times \de_w)=\rho(u e^{2\pi i\si(\cdot,w)}).
$$
Note  that also $\rho(u e^{2\pi i\si(\cdot,w)})$ is thus a bounded operator.

Moreover, since $\rho(\phi)=T_{\phi_1\otimes \phi_2},$ one checks that  $P\rho(\phi)P=T_{\check \phi_1\otimes\check \phi_2},$ and thus 
\begin{eqnarray*}
T\star\rho(\phi)(w)&=&\tr[\rho(u e^{2\pi i\si(\cdot,w)}) \,T_{\check \phi_1\otimes\check \phi_2}]\\
&=&\tr[T_{(\rho(u e^{2\pi i\si(\cdot,w)}) \, \check \phi_1)\otimes\check \phi_2}]\\
&=& \int (\rho(u e^{2\pi i\si(\cdot,w)})\check \phi_1)(x)\, \check \phi_2(x)dx\\
&=& \langle k_{(u e^{2\pi i\si(\cdot,w)})}, \check \phi_2\otimes \check \phi_1)\rangle,
\end{eqnarray*}
where we have also used \eqref{tracetensor}.
But, one easily  computes  that 
$$
A(\check\phi_2,\check \phi_1)=P_\xi A(\phi_1,\phi_2),
$$
so that, in combination with \eqref{cP},
$$
T\star\rho(\phi)(w)=\langle k_{(u e^{2\pi i\si(\cdot,w)})}, k_{P_\xi A(\phi_1,\phi_2)}\rangle=\langle u e^{2\pi i\si(\cdot,w)}, A(\phi_1,\phi_2)\rangle,
$$
i.e.,
$$
T\star\rho(\phi)(w)=\langle u e^{-2\pi i\si(w,\cdot)},\phi\rangle=\F_\si(\phi u)(w).
$$
This proves \eqref{star2} for the case where   $T=T_{\phi_1\otimes \phi_2},$ and hence concludes the proof.
\end{proof}

The following corollary to Lemma \ref{wernerg} is immediate by Young's inequality in Proposition \ref{youngw} for Werner's convolution product, in combination with  Lemma \ref{S1est}:
\medskip

\begin{cor}\label{inversest}
Let $T=\rho(u)\in \S^p,$ where $1\le p\le\infty,$ and let  $F\in\S(\RR^{2d}).$ Then 
$$
\|\F_\si(F u)\|_{L^p(\RR^{2d})}\le \|\rho(F)\|_{\S^1}\|T\|_{\S^p}\le C \|\F_\si(F)\|_{(4d+2,2)}\, \|L_{\F_\si u}\|_{\S^p}.
$$
 \end{cor}
 
For the sake of completeness, let us finally also give a short proof of  the following well-known result:
\begin{lemma}\label{rhoSdense}
$\rho(\S(\RR^{2d}))$ is dense in $\S^p$ for $1\le p<\infty.$
\end{lemma}

\begin{proof} 
If $T\in\S^p\subset \K,$ then let $\{s_l\}_{l\in \NN}$ denote the non-increasing sequence of singular values $s_l\ge 0$ of $T,$ and let 
\begin{equation}\label{singval}
T=\sum\limits_{l=0}^\infty s_l\, T_{\vp_l\otimes \overline{\psi_l}}
\end{equation}
be a singular value decomposition of $T,$ where the $\{\vp_l\}_l$ and the $\{\psi_l\}_l$ form orthonormal bases of the Hilbert space $L^2(\RR^d)$ (see e.g., \cite{Scha},\cite{Si}).Then
$$
\big(\sum\limits_{l=0}^\infty s_l^p\big)^{\frac 1p}=\|T\|_{\S^p}<\infty.
$$
If we put  $T_n:=\sum\limits_{l=n}^\infty s_l\, T_{\vp_l\otimes \overline{\psi_l}}, \ n\in \NN,$ then 
$$
\|T-T_n\|_{\S^p}=(\sum\limits_{l=n+1}^\infty s_l^p)^{1/p}\to 0\qquad\text{as}\  n\to \infty.
$$ 
Since  $(\S^p,\|\cdot\|_{\S^p})$ is a Banach space, we thus see that it will suffice to prove that  every rank-one operator $T_{f\otimes g},$ with $f,g\in L^2(\RR^d),$ lies in the closure of $\rho(\S(\RR^{2d}))$ in $\S^p.$ Note that 
$$
\|T_{f\otimes g}\|_{\S^p}=\|f\|_{L^2}\|g\|_{L^2}.
$$

Thus, if we choose sequences $\{f_n\}_n, \{g_n\}_n$ in $\S(\RR^d)$ such that $f=\lim f_n$ and $g=\lim g_n$ in $L^2,$
then this implies that $T_{f\otimes g}=\lim T_{f_n\otimes g_n}$ in $\S^p.$ Moreover, since $f_n\otimes g_n\in\S(\RR^{2d}),$ 
our previous discussion show that $T_{f_n\otimes g_n}\in \rho(\S(\RR^{2d}).$ Our claim follows.
\end{proof}


\section{Schatten class estimates for pseudo-differential operators in the Weyl calculus}\label{Schaest}
Our key result is the following

\begin{thm}\label{Taest}
Let $\vp\in \S(\RR^{2d}),$ and  define for any $a\in  \S'(\RR^{2d})$ the linear operator 
$$M_a:=L_{a*\vp}.$$
Then there is a constant $C>0$,  such that  for every $1\le p\le\infty,$ 
\begin{equation}\label{tain}
\|M_a\|_{\S^p}\le C\Big( \sum\limits_{|\al|+|\beta|\le 4d+2}\|\pa^\al (z^\beta \F_\si(\vp)\|_{L^2(\RR^{2d})}\Big)
\|a\|_{L^p(\RR^{2d})}.
\end{equation}
\end{thm}

\begin{proof}

Consider first the case $p=\infty.$ Since  $\S^\infty=\B(L^2(\RR^d)),$ we have $\|M_a\|_{\S^\infty}=\|M_a\|.$ 
Moreover, $\pa^\al(a*\vp)=a*(\pa^\al\vp),$ so that
$$
\|\pa^\al(a*\vp)\|_\infty\le \|\pa^\al\vp\|_1\|a\|_\infty.
$$
Thus, by Theorem \ref{CaVa}, 
\begin{equation}\label{sinfty}
\|M_a\|_{S^\infty}=\|L_{a*\vp}\|\le  C\|a*\vp\|_{C^{2d+1}}\le C\sum\limits_{|\al|\le 2d+1}  \|\pa^\al\vp\|_{L^1(\RR^{2d})}\|a\|_\infty.
\end{equation}

\medskip
Assume next that $p=1.$ Since $a*\vp(z)=\int a(w)\vp(z-w) dw,$ i.e., 
$$
a*\vp=\int a(w)\la_w\vp\, dw,
$$
(as  an $L^1(\RR^{2d})$-valued Bochner integral), we have
$$
L_{a*\vp}=\int a(w)L_{\la_w\vp} \, dw
$$
as  an $S^1$-valued Bochner integral, 
 and thus 
$$
\|L_{a*\vp}\|_{\S^1}\le\int |a(w)| dw \cdot \sup\limits_{w\in \RR^{2d}}\|L_{\la_w\vp}\|_{\S^1}.
$$
By Lemma \ref{Sptrans} and  Lemma \ref{S1est}, we thus get
\begin{equation}\label{s1}
\|M_a\|_{S^1}\le \|a\|_{L^1}\|L_\vp\|_{\S^1}\le C\|\F_\si(\vp)\|_{(4d+2,2)}\,\|a\|_1.
\end{equation}

Next, it follows by standard arguments (using Cauchy-Schwarz' inequality  on $\pa^\al \vp=\langle x\rangle^{-m} (\langle x\rangle^{m}\pa^\al \vp),$
with $m\in\{d+1,d+2\}$ chosen so that $m$ is even) that 
$$
\sum\limits_{|\al|\le 2d+1}  \|\pa^\al\vp\|_{L^1(\RR^{2d})}\le C \|\vp\|_{(4d+2,2)}.
$$
Note here that $2d+1+d+2\le 4d+2.$ And, again by Plancherel's theorem, we can show that  $\|\vp\|_{(4d+2,2)} \simeq \|\F_\si(\vp)\|_{(4d+2,2)}. $ Thus we may finally obtain  from \eqref{sinfty} an estimate analogous  to \eqref{s1}:
\begin{equation}\label{sinfty2}
\|M_a\|_{S^\infty}\le C \|\F_\si(\vp)\|_{(4d+2,2)}\; \|a\|_\infty.
\end{equation}

\smallskip

Estimate \eqref{tain} now follows from \eqref{s1} and \eqref{sinfty2} by complex  interpolation.
\end{proof}

The  next  theorem is our main result.  It is an easy consequence of Theorem \ref{Taest} and Corollary \ref{inversest}. Note that it greatly extends and improves on the last statements in Theorem 1.4 in \cite{LSa}, which only apply to compactly supported non-negative Radon measures $\mu,$ whereas we can consider arbitrary compactly supported distributions $u.$

\begin{thm}\label{schattenests}
There is a constant $C>0,$ such that the following hold true:

If $u\in\E'(\RR^{2d})$ is a distribution with compact support contained in a closed ball $\overline {B(z_0,R)}$ of radius $R\ge 1,$ then for  $1\le p\le \infty,$ 
\begin{equation}\label{schattest1}
\|L_{\F_\si(u)}\|_{\S^p}\le C R^{5d+2} \|\F_\si(u)\|_{L^p(\RR^{2d})},
\end{equation}
and, conversely, 
\begin{equation}\label{schattest2}
\|\F_\si(u)\|_{L^p(\RR^{2d})}\le C R^{5d+2} \|L_{\F_\si(u)}\|_{\S^p}.
\end{equation}

\end{thm}
\begin{proof} To prove \eqref{schattest1}, let us first consider the case $z_0=0.$ We then choose $\chi\in C^\infty_0(\RR^{2d})$ so that $\chi=1$ on $\overline {B(0,2)},$ 
and put $\chi_R(z):=\chi(z/R).$ Then $u=\chi_R u,$ so that
$$
\F_\si(u)=\F_\si(u)*\vp,
$$
with $\vp:=\F_\si(\chi_R).$ 
Thus, by Theorem \ref{Taest}, we can estimates 
$$
\|L_{\F_\si(u)}\|_{\S^p}\le C\|\F_\si(\F_\si(\chi_R)\|_{(4d+2,2)}\|\F_\si(u)\|_{L^p(\RR^{2d})}=C\|\chi_R\|_{(4d+2,2)}\|\F_\si(u)\|_{L^p(\RR^{2d})}.
$$

And, if $|\al|+|\beta|\le 4d+2,$ then it is easily seen by Leibniz' rule that 
$$
\|\pa^\al (z^\beta \chi_R)\|_2\lesssim R^{|\beta-\al|+2d/2}\le R^{5d+2},
$$
which proves the claimed estimate for the case $z_0=0.$

\medskip
If $z_0\in\RR^{2d}$ is arbitrary, we consider $v:=\la_{-z_0} u\in \E'.$ Then $v$ is supported in $\overline {B(0,R)}.$ Moreover, by \eqref{proj},
$$
L_{\F_\si(u)}=\rho (u)=\rho(\la_{z_0}v)=\rho(e^{\pi i \si(\cdot, z_0)} v) \rho(z_0),
$$
so that
$$
\|L_{\F_\si(u)}\|_{\S^p}\le \|\rho(e^{\pi i \si(\cdot, z_0)} v) \|_{\S^p}= \|L_{\F_\si(e^{\pi i \si(\cdot, z_0)} v)} \|_{\S^p}.
$$
Thus, by the previous estimate for $z_0=0,$
\begin{eqnarray*}
\|L_{\F_\si(u)}\|_{\S^p}&\le& C R^{5d+2}\|\F_\si(e^{\pi i \si(\cdot, z_0)} v)\|_{L^p(\RR^{2d})}
=C R^{5d+2}\|\F_\si(v)\|_{L^p(\RR^{2d})}\\
&=&C R^{5d+2}\|\F_\si(u)\|_{L^p(\RR^{2d})}.
\end{eqnarray*}

\medskip
Let us next turn to estimate \eqref{schattest2}. In a similar way as before, we may reduce to the case $z_0=0.$ Choosing then $F:=\chi_R$ in Corollary \ref{inversest}, we have $Fu=u,$ so that we may estimate 
$$
\|\F_\si(u)\|_{L^p(\RR^{2d})}\le  C \|\F_\si(\chi_R)\|_{(4d+2,2)}\, \|L_{\F_\si u}\|_{\S^p}.
$$
But, by what we have proved before, $\|\F_\si(\chi_R)\|_{(4d+2,2)}\simeq \|\chi_R\|_{(4d+2,2)}\lesssim R^{5d+2},$ which completes the proof.
\end{proof}

With a bit more effort, we can also prove the following extension of the first statement in Theorem 1.4 of \cite{LSa}, whose proof was based on a Tauberian theorem. Our extension to arbitrary compactly supported distributions of that result avoids  Tauberian arguments and is rather  making  use again of the Calderón-Vaillancourt theorem and Lemma \ref{wernerg}, instead. By 
$C_\infty=C_\infty(\RR^{2d})$ we shall denote the Banach space of all continuous functions on $\RR^{2d}$ which vanish at infinity, endowed with the sup-norm.
\begin{thm}\label{compactL}
Suppose $u\in\E'(\RR^{2d})$ is a distribution with compact support. Then the operator $L_{\F_\si(u)}$ is compact if and only if 
$\F_\si(u)$ vanishes at infinity, i.e., $\F_\si(u)\in C_\infty(\RR^{2d}).$
\end{thm}
\begin{proof} 
Assume first that $\F_\si(u)\in C_\infty(\RR^{2d}).$ Choosing  $\chi\in C^\infty_0(\RR^{2d})$ so that $\chi=1$ on a neighborhood of  the support of $u,$ we  see that  $\F_\si(u)=\F_\si(u)*\F_\si(\chi).$ Next, we fix a Dirac family $\{\psi_\ve\}_{\ve>0}$ in  $C^\infty_0(\RR^{2d})$ of the form $\psi_\ve(z)=\ve^{-2d}\psi(z/\ve),$ where $\psi\in C^\infty_0(\RR^{2d})$ is supported in the unit ball $\overline {B(0,1)}$ and  $\psi(z)=1$ for $|z|\le 1/2,$ and put 
$$
u_\ve:=u*\psi_\ve.
$$
Then $u_\ve\in\D,$ and 
$$
\F_\si u-\F_\si u_\ve=(\F_\si u)(1-(\F_\si\psi)(\ve\,\cdot))=(\F_\si(u)*\F_\si(\chi))(1-(\F_\si\psi)(\ve\,\cdot)).
$$
By Leibniz' rule, we get
$$
\pa^\al(\F_\si u-\F_\si u_\ve)=\sum\limits_{0\le\beta\le \al}\binom \al \beta\big((\F_\si(u)*\pa^{\al-\beta}\F_\si(\chi)\big)\pa^\beta(1-(\F_\si\psi)(\ve\,\cdot)).
$$
Note that by our assumption on $\F_\si(u),$  $\F_\si(u)*\pa^{\al-\beta}\F_\si(\chi)\in C_\infty(\RR^{2d})$ for any $\al,\beta\in\NN^{2d},$ since  $\pa^{\al-\beta}\F_\si(\chi)\in\S.$ This easily implies that for every $\al\in\NN^{2d},$ 
$$
\|D^\al(\F_\si u-\F_\si u_\ve)\|_\infty \to 0\qquad\text{as}\ \ve\to 0.
$$
Thus, by Calder\'on-Vaillancourt's Theorem \ref{CaVa}, we conclude that
\begin{equation}\label{approxL}
\|L_{\F_\si(u)}-L_{\F_\si(u_\ve)}\|\to 0\qquad\text{as}\ \ve\to 0.
\end{equation}
But $L_{\F_\si(u)}\in\B(L^2(\RR^d))$ by Theorem \ref{schattenests}, and $L_{\F_\si(u_\ve)}\in\S^1\subset\K$ by Lemma \ref{S1est}, and since $\K$ is a closed subspace  of $\B(L^2(\RR^d)),$ this implies that also $L_{\F_\si(u)}\in\K.$

\medskip
Conversely, assume now that $T:=L_{\F_\si(u)}=\rho(u)\in\K.$ We then show that $\F_\si(u)\in C_\infty(\RR^{2d}).$
Since $u$ has compact support, $\F_\si(u)$ is real analytic, hence in particular continuous. Moreover, choosing $\phi:=\chi$ in Lemma \ref{wernerg}, we see that 
\begin{equation}\label{Fsiu}
\F_\si(u)(w)=\F_\si(\chi u)(w)=T\star\rho(\chi)(w)=\tr[T^{\rho(-w)}\, S],
\end{equation}
where we have set $T^{\rho(-w)}:=\rho(-w)T\rho(w)$ and $S:=P\rho(\chi)P.$ Note that, by Lemma \ref{S1est}, $S\in \S^1,$ since 
$\rho(\chi)=L_{\F_\si(\chi)},$ with $\F_\si(\chi)\in\S(\RR^{2d}).$

Consider  the bilinear mapping $\beta:(A,B)\mapsto F_{A,B}$ from $\B(L^2(\RR^d))\times \S^1\to C_b(\RR^{2d}),$ defined by
$$
F_{A,B}(w):= \tr[A^{\rho(-w)}\, B],\qquad w\in\RR^{2d},
$$
where $ C_b(\RR^{2d})$ denotes the Banach space of all bounded continuous functions on $\RR^{2d},$ endowed with the sup-norm $\|\cdot\|_\infty.$ Since by\eqref{ATB}
\begin{equation}\label{betaest}
\|F_{A,B}\|_\infty=\sup_w |F_{A,B}(w)|\le \|A\| \|B\|_{\S^1},
\end{equation}
we see that $\beta$ is continuous. Moreover, by \eqref{Fsiu}, $\F_\si(u)=F_{T,S}.$ 

Let next $T=\sum\limits_{l=0}^\infty t_l\, T_{\vp_l\otimes \overline{\psi_l}}$ and $S=\sum\limits_{l=0}^\infty s_l\, T_{g_l\otimes \overline{h_l}}$ be singular value decompositions of $T$ and $S,$ respectively, and let, for $n\in\NN,$  
$$
T_n=\sum\limits_{l=n}^\infty t_l\, T_{\vp_l\otimes \overline{\psi_l}},\quad S_n=\sum\limits_{l=n}^\infty s_l\, T_{g_l\otimes \overline{h_l}}.
$$
Since $t_l\to 0$ as $l\to\infty,$ we see that $\|T-T_n\|\to 0$ as $n\to \infty,$ and, as shown in the proof of Lemma \ref{rhoSdense}, 
$\|S-S_n\|_{\S^1}\to 0$ as $n\to \infty.$ By \eqref{betaest}, this implies that $F_{T_n,S_n}$ converges uniformly to $F_{T,S}$
as $n\to\infty,$ and thus it will suffice to show that $F_{T_n,S_n}\in C_\infty$ for every $n.$ This, in return, reduces to showing that 
$F_{T_{\vp\otimes \overline{\psi}},T_{g\otimes \overline{h}}}\in C_\infty$ for all $\vp,\psi,g,h\in L^2(\RR^d).$
But, by approximating these functions $\vp,\psi,g,h$ arbitrarily closely by Schwartz functions as in the proof of Lemma \ref{rhoSdense} and making again use of \eqref{betaest}, it will eventually suffice to prove that 
$F_{T_{\vp\otimes \overline{\psi}},T_{g\otimes \overline{h}}}\in C_\infty$ for all $\vp,\psi,g,h\in \S(\RR^d).$
\smallskip

But, one computes that $T_{\vp\otimes \overline{\psi}}^{\rho(-w)}=T_{\rho(-w)\vp\otimes \overline{\rho(-w)\psi}},$ hence, by 
\eqref{crosFW}, 
\begin{eqnarray*}
F_{T_{\vp\otimes \overline{\psi}},T_{g\otimes \overline{h}}}(w)&=&(\rho(w)g,\psi)(\vp,\rho(w)h)=\overline{A(\psi,g)}(w)A(\vp,h)(w)\\
&=&\overline{\cP^{-1}(w)(\psi\otimes g)}\cP^{-1}(h\otimes g)(w).
\end{eqnarray*}
Since $\psi\otimes g$ and $h\otimes g$ are assumed to be Schwartz functions, the same is true of $\overline{\cP^{-1}(\psi\otimes g)}$ and $\cP^{-1}(h\otimes g),$ and thus $F_{T_{\vp\otimes \overline{\psi}},T_{g\otimes \overline{h}}}\in \S\subset C_\infty.$ 
This concludes the proof of the theorem.
\end{proof} 

\section{Fourier restriction for Schatten classes}\label{frschatten}

The main Theorem 1.2 in \cite{LSa} is essentially an immediate consequence of  Theorem \ref{schattenests}, which also leads to sharper bounds with respect to the diameter of the support of the measure $\mu.$ 
\smallskip

Let us first recall the following definitions: for a trace class operator $S$ on $L^2(\RR^d),$ its {\it Fourier-Wigner transform} $\F_WS,$  also called 
{\it Fourier-Weyl transform}, is defined to be the function
$$
\F_WS(z):= \tr(\rho(z)^*S), \qquad z\in\RR^{2d}.
$$
The  formally adjoint operator to the {\it Fourier restriction operator} $g\mapsto \F_\si(g)|_{\supp \mu}$ is the {\it Fourier extension operator} 
$$
\E_\si(f):=\F_\si(fd\mu).
$$
The analogous extension operator formally adjoint to the operator $\F_W$ is the {\it quantum extension operator} 
$$
\E_W(f):=\rho(f d\mu)=\int_{\RR^{2d}} f(z) \rho(z)  d\mu(z),
$$
where the integral is to be understood in the weak operator topology sense.

\begin{thm}[Luef and Samuelsen]\label{restrictionS}
Let $\mu\in \M^1(\RR^{2d})$ be a non-negative bounded Radon measure with compact support. Then, for $1\le p,q\le\infty,$ with conjugate exponents $p',q',$ the following statements are equivalent:
 \begin{itemize}
\item[i)] There exists a constant $C_\si>0$ such that for any $g\in L^p(\RR^{2d}),$ 
$$
\|\F_\si(g)\|_{L^q(\mu)}\le C_\si\|g\|_{L^p(\RR^{2d})}.
$$
\item[ii)] There exists a constant $C_W>0$ such that for any $T\in \S^p,$ 
$$
\|\F_W(T)\|_{L^q(\mu)}\le C_W\|T\|_{\S^p}.
$$
\item[iii)] There exists a constant $C_\si>0$ such that for any $f\in L^{q'}(\mu),$ 
$$
\|\E_\si(f)\|_{L^{p'}(\RR^{2d})}\le C_\si \|f\|_{L^{q'}(\mu)}.
$$
\item[iv)]  There exists a constant $C_W>0$ such that for any $f\in L^{q'}(\mu),$
$$
\|\E_W(f)\|_{\S^{p'}}\le C_W \|f\|_{L^{q'}(\mu)}.
$$
\end{itemize}
\end{thm}

\begin{proof} 
Our previous results allow for a very short proof of the equivalence of iii) and iv). Observe first that the estimates in iii) and iv) always hold true in a trivial way when $p'=\infty,$ since for $f\in L^{q'}(\mu)$ we have $\|f d\mu\|_{\M^1}\lesssim \|f\|_{L^{q'}(\mu)},$ by H\"older's inequality, and since $\F_\si(\M^1)\subset L^\infty$ and  $\rho(\M^1)\subset \S^\infty.$

\smallskip We may therefore assume that $p'<\infty.$ But, choosing $u:=fd\mu$ in Theorem \ref{schattenests}, we see that if $\mu$ is supported in a closed ball of radius $R\ge 1,$ then  
$$
\|\E_W(f)\|_{S^{p'}}\simeq_R \|\F_\si(f d\mu)\|_{L^{p'}(\RR^{2d})}=\|\E_\si(f)\|_{L^{p'}(\RR^{2d})}.
$$
The equivalence of iii) and iv) is now immediate.

\medskip
The equivalence of i) with iii) is a classical result, which follows easily by duality, since 
$$
\int \overline{\F_\si(g)}f d\mu=\int \overline g\, \E_\si(f) dz,
$$
and since on a general measure space the following holds true:
$$
\|F\|_{L^{p'}}=\sup\limits_{\|G\|_p\le 1}|( F,G)|, \qquad 1\le p\le \infty.
$$

Similarly, we have 
$$
\int \overline{\F_W(T)(z)} f(z) d\mu(z)=\tr (\E_W(f) T^*), 
$$
and 
$$
\|S\|_{\S^{p'}}=\sup\limits_{\|T\|_{\S^p}\le 1}|\tr (ST^*)| \qquad 1\le p\le \infty
$$
(see, e.g., \cite{RS} for the cases $p'=1,\infty,$ and  \cite{Siop} for $1<p'<\infty$).
This allows to easily prove the equivalence of ii) and iv).
\end{proof} 

\begin{remark}\label{improveconst}
An analysis of the proof of Theorem 3.1 in \cite{LSa} shows that if $\mu$ is supported in a closed ball of radius $R\ge 1,$ and if  i) holds true, then ii) holds  with a constant
$C_W\le e^{\pi R^2/2} C_\si,$ and if ii) holds true, then i) holds  with a constant
$C_\si\le e^{\pi R^2/2} C_W.$ 

Our proof, based on Theorem \ref{schattenests}, gives sharper estimates, as it allow to replace the  reciprocal Gaussian factor $e^{\pi R^2/2}$ by a factor $CR^{5d+2},$ which grows only polynomially in $R.$  
\end{remark}


\end{document}